\documentclass[12pt,reqno,a4paper]{amsart}

\usepackage{euscript}
\usepackage{amsmath,amssymb,amsfonts}
\usepackage{dsfont}

\newtheorem{theorem}{Theorem}
\newtheorem{proposition}[theorem]{Proposition}
\newtheorem{lemma}[theorem]{Lemma}
\newtheorem{cor}[theorem]{Corollary}
\newtheorem{remark}[theorem]{Remark}
\newtheorem{defin}[theorem]{Definition}

\renewcommand{\epsilon}{\varepsilon}
\renewcommand{\L}{\mathcal{L}}
\DeclareMathOperator{\fr}{fr}
\DeclareMathOperator{\rsp}{rsp}
\DeclareMathOperator{\smf}{sf}

\def\R{\mathbb{R}}

\begin{document}
	\title[Fractional susceptibility at Misiurewicz parameters]{Pre-threshold fractional susceptibility functions at Misiurewicz parameters.}
	
	\begin{abstract}
		We show that the response, frozen and semifreddo fractional susceptibility functions of certain real-analytic unimodal families, at Misiurewicz parameters and for fractional differentiation index $0\le\eta<\frac{1}{2}$, are holomorphic on a disk of radius greater than one. This is a step towards solving a Conjecture of Baladi and Smania, in the case of the aforementioned susceptibility functions.    
	\end{abstract}
	
	\author{Julien Sedro}
	\address{Laboratoire de Probabilit\'es, Statistique et Mod\' elisation (LPSM), Sorbonne Universit\' e, Universit\'e de Paris, 4
		Place Jussieu, 75005 Paris, France}
	\email{sedro@lpsm.paris}
	
	\maketitle
	\section{Introduction}
	\iffalse
	Where to go from here ? Studying the real fractional susceptibility function. Needs uniform (in parameters) decay of correlations estimates. Extends to more general type of parameters ?
	
	TO BE DONE: 
	%\\Possible to get regularity of invariant density just by studying Fourier multiplier of Marchaud derivative ?
	%limits when $\eta\to\frac{1}{2}$ ? 
	\\Optimizing radius of holomorphy disk ? Or just loose bound ? 
	\\Show coincidence of response and frozen fractional susceptibility functions (possible problem with $t\not\in(t_{\min},t_{\max})$).
	%\\Clean proof of exponential decay of correlations for Sobolev observables.
	\fi
	
	We consider in this paper the \emph{quadratic family} $(f_t)_{t\in(1,2)}$, classically defined by
	\begin{equation}\label{def:quadfamily}
		f_t(x):=t-x^2
	\end{equation}
	on the interval $I_t=[-\beta_t,\beta_t]$ with $\beta_t=\frac{1+\sqrt{1+4t}}{2}\in(0,2)$. This map admits $c=0$ as its unique critical point; we will denote by $c_{k,t}=f_t^k(0),~k\ge 1$ the post-critical orbit.
	\par It is well-known since the work of Lyubich \cite{Lyu} that for almost every parameter $t\in (1,2)$, one of two behavior occurs: either there is an attracting periodic cycle (at so called \emph{regular} parameters; this concerns a dense subset $\mathcal R$ of $(1,2)$), or there exists an absolutely continuous invariant probability measure $d\mu_t=\rho_t dx$ for $f_t$ (at so called \emph{stochastic} parameters, which form a positive Lebesgue measure subset $\mathcal S$ of $(1,2)$).
	\smallskip
	
	In the setting of unimodal interval maps, Collet and Eckmann \cite{CE} famously introduced the following condition: the quadratic map $f_t$ satisfies the Collet-Eckmann condition if there is $\lambda_c>1$ and $N_c>0$ such that for any $k\ge N_c$
	\begin{equation}\label{CE}
		|(f_t^k)'(c_{1,t})|\ge\lambda_c^k,\tag{CE}
	\end{equation}
	i.e one requires the Lyapunov exponent computed along the post-critical orbit to be positive. This condition implies the existence of an absolutely continuous invariant measure, i.e parameters $t$ satisfying \eqref{CE} are in $\mathcal S$. From now on, we will call them $CE$ parameters.  
	\par Let $\mathcal{M}\subset CE$ be the set of Misiurewicz parameters, that is the set of parameters $t$ for which the critical point $c=0$ is not an accumulation point of the post-critical set. Let $ MT\subset \mathcal{M}$ be the set of Misiurewicz-Thurston parameters, for which the post-critical orbit is preperiodic and hyperbolic, i.e such that there exists $\ell,p\in\mathbb N$, for which $c_{\ell,t}=f_t^{\ell}(0)$ is periodic of period $p$, and $|Df_t^p(c_{\ell,t})|>1$.
	\\$\mathcal{M}$ has zero Lebesgue measure \cite{Sands}, and $MT$ is a countable subset of $CE$.
	\smallskip
	
	The question we want to study here is connected to fractional response of the invariant measure, that is the H\"older regularity of the map $t\mapsto\mu_t$, restricted to a suitable subset\footnote{In particular, regularity should be understood in the sense of Whitney \cite{Whitney}.} of $\mathcal S$, at a point $t\in\mathcal M$, where $\mu_t$ is seen as a Radon measure or a distribution of some finite order. The question of regularity of $t\mapsto\mu_t$ has already received a lot of attention, and partial answers. For the quadratic family \eqref{def:quadfamily}, Thunberg \cite[Cor. 1 and 2]{Thun01} showed that the map $t\mapsto\mu_t$ is discontinuous at every point in $\mathcal S$, and cannot be continuous on any full-measure subset of parameters. However, restricted to a suitable subset, this map is continuous: Tsujii \cite{Tsujii} showed that, restricted to a positive measure parameter subset $\mathcal S'\not=\mathcal S$, $t\mapsto\mu_t$ is weak-$\ast$ continuous at Misiurewicz points. Rychlik and Sorets \cite{RS} showed that for yet another positive measure subset $\mathcal S''$ of parameters, the invariant density $\rho_t\in L^p$ for $1\le p< 2$ and that $t\in \mathcal S''\mapsto \rho_t\in L^p$, $1\le p<2$ is continuous at Misiurewicz parameters, via a H\"older estimate.  
	\\More recently, and for more general smooth unimodal families, Baladi, Benedicks and Schnellmann \cite{BBS} proved that at almost every $CE$ parameter $t_0$, for any $\frac{1}{2}$-H\"older observable $\phi$, the map $R_\phi:t\mapsto\int\phi d\mu_t$ is $C^\eta$ at $t=t_0$ for any $\eta<\frac{1}{2}$, in the sense of Whitney, on a set of $CE$ parameters having $t_0$ as a density point. Furthermore, they show that at any mixing MT parameter, there exists $\phi\in C^\infty$, $C>1$, and $t_n\in MT$ with $t_n\to t$, such that
	\[C^{-1}|t-t_n|^{1/2}\leq\left|R_\phi(t_n)-R_\phi(t)\right|\leq C|t-t_n|^{1/2}.\]
	Interestingly, those results seem to contradict earlier ones obtained by Ruelle \cite{Ru05,RuJi05}, who suggested that, for the quadratic family, the map $R_\phi$ had a well-defined derivative at $t=t_0$, for $t_0$ a $MT$ parameter, which raised the hope that \emph{linear response} (i.e. differentiability of the map $t\mapsto\mu_t$) holds in this setting. Note that this hope was already diminished by a series of paper \cite{Bal07,BS08,SdL}, which exhibited  smooth families of piecewise expanding unimodal maps for which linear response fails, and highlighted a sufficient condition, called \emph{horizontality} of the perturbation, for linear response to hold. For a full account of this intricate story, we refer to the Introduction of \cite{ABLP20,BS20}.
	
	\par	
	In a recent preprint, Aspenberg, Baladi, Leppanen and Persson \cite{ABLP20} introduced several \emph{fractional susceptibility functions}, whose connection to fractional response is similar to the one between classical susceptibility function and linear response (i.e, the value at 1 of the (fractional) susceptibility function, if it is well-defined, is the (fractional) derivative of $R_\phi$ at $t=t_0$). We recall here their definitions: given $\phi\in L^\infty(I_{t_0})$, and denoting $\L_t$ the (Ruelle-Perron-Frobenius) transfer operator associated to $f_t$, we define, as formal power series
	\begin{itemize}
		\item The \emph{response} fractional susceptibility function\footnote{In \cite[Def 2.3]{BS20}, another definition of the response fractional susceptibility is given. We note that the two definitions coincide for $0\le\eta<1/2$ and $\phi\in C^1$ compactly supported (see also \cite[Lemma 5.2]{BS20})}:
		\begin{equation}\label{def:respfracsus}
			\Psi^{\rsp}_\phi(\eta,z):=-\sum_{j=0}^\infty z^j\int_{I_{t_0}}\phi\circ f^j_{t_0}(x)M^\eta[\rho_{t_0}](x)dx,
		\end{equation}
		where we denoted $M^\eta$ the two-sided \emph{Marchaud derivative}\footnote{Marchaud fractional derivatives were introduced in his PhD thesis \cite{Marchaud}; for a quick introduction to the subject, we refer to \cite{ferrari}} (see Definition \ref{def:Marchaudderiv} and \eqref{eq:twosidedMarchaud}).
		\item The \emph{frozen} fractional susceptibility function:
		\begin{equation}\label{def:frozenfracsus}
			\Psi^{\fr}_\phi(\eta,z):=\sum_{j=0}^\infty z^j\int_{I_{t_0}}\phi\circ f^j_{t_0}(x)M_t^\eta[\L_{t}\rho_{t_0}]_{t=t_0}dx,
		\end{equation}
		where we denoted $M_t^\eta$ the two-sided Marchaud derivative w.r.t the parameter $t$.
	\end{itemize}
	To define a fractional susceptibility function in the spirit of \cite{ABLP20}, more care is needed. Indeed, as the invariant density $\rho_t$ is not defined for every parameter $t$, one has to consider integrals over some positive measure subset $\Omega\subset\mathcal S$. Then one sets
	\begin{small}
		\begin{equation}\label{def:fracsus}
			\Psi^{\Omega}_\phi(\eta,z):=\frac{\eta}{2\Gamma(1-\eta)}\sum_{j=0}^\infty z^j\int_{-2}^2\int_{\Omega-t_0}\phi\circ f^j_{t+t_0}\frac{(\L_{t_0+t}-\L_{t_0})\rho_{t_0}}{|t|^{1+\eta}}\text{sgn}(t)dtdx.
		\end{equation}
	\end{small}
	As an intermediary object between the frozen and real fractional susceptibility, one may introduce the \emph{semifreddo} susceptibility function (see \cite[Section 7.2]{BS20}), defined by
	\begin{small}
		\begin{equation}\label{def:sffracsus}
			\Psi^{\smf}_\phi(\eta,z):=\frac{\eta}{2\Gamma(1-\eta)}\sum_{j=0}^\infty z^j\int_{I_{t_0}}\int_{\Omega-t_0}\phi\circ f^j_{t_0}\frac{(\L_{t_0+t}-\L_{t_0})\rho_{t_0}}{|t|^{1+\eta}}\text{sgn}(t)dtdx.
		\end{equation}
	\end{small}
	In the context of the quadratic family, Baladi and Smania \cite[Conjecture A]{BS20} formulate a set of conjectures for those fractional susceptibility functions associated with compactly supported $C^1$ observables $\phi$: holomorphy in $z$ in a disk of radius greater than one and fractional response formula\footnote{This property is only valid for the ``real" fractional susceptibility function \eqref{def:fracsus}} for $0<\eta<1/2$, existence of a certain decomposition at $\eta=\frac{1}{2}$, holomorphy in a disk of radius smaller than one for $\frac{1}{2}<\eta<1$. Note the qualitative change of behavior when $\eta$ passes the value $1/2$: this is the reason of our ``threshold" terminology. 
	\\In particular, \cite[Thm. C]{BS20} establishes the existence of some decomposition, related to the presence of poles on the unit circle,  for the response and frozen fractional susceptibilities at the threshold value $\eta=\frac{1}{2}$. 
	\medskip
	
	As a step towards proving those conjectures, we establish in Theorem \ref{thm:regfrozenfracsus}, for mixing Misiurewicz parameters, holomorphy of \eqref{def:respfracsus}, \eqref{def:frozenfracsus} and \eqref{def:sffracsus} in a disk of radius greater than one, for $0\le \eta<\frac{1}{2}$. In light of the previous discussion, this result is the best one can expect in this setting.  
	\medskip
	
	Given the global nature of the Marchaud fractional derivative and \eqref{def:frozenfracsus}, it will be useful to extend the range of parameters $t\in(1,2)$ of the quadratic family to $\R$, as follows: for a $t_0\in(1,2)$, fix a $t_{min}<0$ (resp. a $t_{max}>0$) such that $t_0+t_{min}>1$ (resp. $t_0+t_{max}<2$), and set $f_t:=f_{t_0+t_{min}}$ for $t\le t_0+t_{min}$ and $f_t:=f_{t_0+t_{\max}}$ for all $t\ge t_0+t_{max}$.
	\medskip
	
	\begin{remark}\label{rem:gensetting}
		We point out that our results are stated and proven in a more general setting than the quadratic family: let $(f_t)_{t\in[t_{min},t_{max}]}$ be a family of real-analytic, unimodal maps, with negative Schwarzian derivative, satisfying at a Misiurewicz parameter $t_0\in[t_{min},t_{max}]$:
		\begin{equation}\label{eq:gensetting}
			f_{t+t_0}= f_{t_0}+tX_{t_0}\circ f_{t_0},
		\end{equation} 
		with $X_{t_0}$ real-analytic. %as well as the \emph{transversality} condition 
		%\[\sum_{k=0}^\infty\frac{X_{t_0}(c_{k,t_0})}{D[f^k_{t_0}](c_{1,t_0})}\not= 0. \]
		Note that for the quadratic family, \eqref{eq:gensetting} holds with $X_{t_0}\equiv 1$. %and that in this setting, all $CE$ parameters (hence all Misiurewicz parameters) are transversal.
	\end{remark}
	\medskip
	
	\textbf{Acknowledgments}:
	The author was supported by the European Research Council (ERC) under the European Union's Horizon 2020 research and innovation programme (grant agreement No 787304). The author would also like to thank Viviane Baladi for suggesting the problem and many useful conversations, especially for the proof of Lemma \ref{lemme:decayofcorrel}.
	\section{Preliminaries}
	In this section, we recall useful definitions and results.
	We start with:
	\begin{defin}\label{def:fracint}
		Let $0<\eta<1$ and $1\le p<\frac{1}{\eta}$. For $g\in L^p(\mathbb R)$, we define its \emph{fractional integrals} $I_{\pm}^\eta[g]$ by:
		\begin{align}
			I_{-}^\eta[g](x)&:=\frac{1}{\Gamma(\eta)}\int_{x}^{+\infty}\frac{g(t)}{(x-t)^{1-\eta}}dt\\
			I_{+}^\eta[g](x)&:=\frac{1}{\Gamma(\eta)}\int_{-\infty}^{x}\frac{g(t)}{(t-x)^{1-\eta}}dt
		\end{align} 
	\end{defin}
	Parallel to those fractional integrals, we introduce the (Marchaud) fractional derivatives: 
	\begin{defin}\label{def:Marchaudderiv}
		Given $0<\eta<1$ and a bounded, $\alpha$-H\"older ($\alpha>\eta$) $g$ on the real line\footnote{The fractional derivative operators $M_{\pm}^\eta$ may be defined for a larger class of functions. For more details, we refer to \cite[Sec. 5.4]{Samko}.}, we define its (left and right) $\eta$-\emph{Marchaud fractional derivative} by
		\begin{align}
			M_{-}^\eta[g](x)&:=\frac{\eta}{\Gamma(1-\eta)}\int_x^{+\infty}\frac{g(x)-g(t)}{(t-x)^{1+\eta}}dt\\
			M_{+}^\eta[g](x)&:=\frac{\eta}{\Gamma(1-\eta)}\int_{-\infty}^x\frac{g(x)-g(t)}{(x-t)^{1+\eta}}dt.
		\end{align} 
		We may then define the \emph{two-sided Marchaud derivative} $M^\eta$ as
		\begin{equation}\label{eq:twosidedMarchaud}
			M^\eta[g]:=\frac{M_{+}^\eta[g]-M_{-}^\eta[g]}{2}
		\end{equation}
	\end{defin}

	As in the case of classical differentiation and integration, fractional differentiation is the left inverse of fractional integration (see \cite[Sec 5.4 and Thm 6.1]{Samko}).
	\\We also introduce fractional integrals on an interval $(a,b)$, $I^\eta_{a_+}[g]$ defined for $0<\eta<1$ and (say) bounded $g:(a,b)\to\R$ by
	\begin{align}\label{def:intervalfracint}
		I_{b_-}^\eta[g](x)&:=\frac{1}{\Gamma(\eta)}\int_{x}^{b}\frac{g(t)}{(x-t)^{1-\eta}}dt\\
		I^\eta_{a_+}[g](x)&:=\frac{1}{\Gamma(\eta)}\int_{a}^{x}\frac{g(t)}{(t-x)^{1-\eta}}dt
	\end{align}
	
	\begin{defin}
		Given $0\le s<1$ and $p>1$, denoting by $\mathcal F$ the Fourier transform, we consider $\phi\in C_c^\infty(\R)$ and set 
		\[\|\phi\|_{H^s_p}:=\|\mathcal F^{-1}\left((1+|\xi|^2)^{s/2}\mathcal F(\phi)\right)\|_{L^p}.\]
		We then define $H^s_p(\R):= \overline{C_c^\infty(\R)}^{\|.\|_{H^s_p}}$.
		\\Given an interval $I\subset\R$, we will also consider the space $H^s_p(I)$ of functions $f\in H^s_p(\R)$ that are supported in $I$.
	\end{defin}
	\begin{defin}
		Let $-\infty<a<b<+\infty$. We say that $f\in I^s_{a+}(L^p(a,b))$ (resp. $f\in I^s_{b-}(L^p(a,b))$) if there exists some $\phi\in L^p(a,b)$ such that $f=I^s_{a+}[\phi]$ (resp. $f=I^s_{b-}[\phi]$). 
		\\We define similarly $I^s_+(L^p(\R))$ and $I^s_-(L^p(\R))$.
	\end{defin}
	One has \cite[Cor to Thm 11.4 and 11.5]{Samko}, for $1<p<\frac{1}{s}$: \[I^s_{a+}(L^p(a,b))=I^s_{b-}(L^p(a,b))=:I^s(L^p(a,b))\] (resp. $I^s_+(L^p(\R))=I^s_-(L^p(\R))=:I^s(L^p(\R))$).
	\\We will make extensive use of the following fact: if a $L^p$ function is representable as the (left or right) $s$-fractional integral of some $\phi\in L^p$, $1<p<\frac{1}{s}$, then it belongs to $H^s_p$. More precisely (see \cite[Corollary to Thm 18.2]{Samko}):
	\begin{theorem}\label{thm:fracintfracsobolev}
		Let $0\le s<1$ and $1<p<\frac{1}{s}$. One has
		\begin{align*}
			&H^s_p(\R)= L^p(\R)\cap I^s(L^p(\R)).
		\end{align*} 
	\end{theorem}
	We state a useful lemma, taken from \cite[Lemma 2.39, p.56]{Baladibook}
	\begin{lemma}\label{lemma:regshiftop}
		Let $0\leq\tilde s<s\leq 1$, and $1<p<\infty$. For $g\in H^s_p(\mathbb R)$, and $T\in C^1$, one has
		\begin{equation*}
			\|g\circ T-g\|_{H^{\tilde s}_p}\leq Cd_{C^1}(T,Id)^{s-\tilde s}\|g\|_{H^s_p}, 
		\end{equation*}
		with $C$ depending boundedly on $\|T\|_{C^1}$.
	\end{lemma} 
	Finally, we recall how the Marchaud derivative acts on the scale of Sobolev spaces: this is a consequence of Theorem \ref{thm:fracintfracsobolev} and \cite[Thm 5.3]{Samko}.
	\begin{proposition}\label{prop:Marchaudboundedop}
		Let $s>0$, $0\le\eta<s$ and $1<p<\frac{1}{s}$. Then \[M^\eta_{\pm}(H^s_p(\R))\subset L^q(\R)\cap I^{s-\eta}(L^p(\R)),\]
		with $q$ such that $\frac{1}{q}-\frac{1}{p}=\eta$. 
	\end{proposition}
	\begin{remark}
		Let us mention that there exists several other notions of fractional derivatives (Bessel fractional potential or the more classical Riemann-Liouville derivative, to name a few), that may appear more natural in different contexts. However, we believe that the result one would obtain using those notions of fractional derivatives are qualitatively similar to ours. Indeed, Proposition \ref{prop:Marchaudboundedop}, a key ingredient in our approach, also holds for other types of fractional derivatives (by construction in the case of the Bessel fractional potential and fractional Sobolev spaces). 
	\end{remark}
	\section{Marchaud derivative of the invariant density at Misiurewicz parameters}
	Recall that at a stochastic parameter $t_0$, the quadratic map $f_{t_0}$ (or more generally, a unimodal real-analytic map with negative Schwarzian derivative) admits a unique absolutely continuous invariant probability measure $d\mu_{t_0}:=\rho_{t_0}dx$. 
	\\	Denoting by $(c_k)_{k\geq 0}$ the critical orbit, Ruelle proved \cite[Theorem 9]{Ruelle} the following decomposition for the invariant density $\rho_{t_0}$ of a Misiurewicz real-analytic unimodal map:
	\begin{small}
		\begin{equation}\label{eq:invdensdec}
			\rho_{t_0}(x)=\psi_0(x)+\sum_{k=1}^\infty C_{k,0}\frac{\mathds 1_{w_0>\sigma_k (x-c_k)>0}}{\sqrt{\sigma_k(x-c_k)}}+C_{k,1}\mathds 1_{w_1>\sigma_k(x-c_k)>0}\sqrt{\sigma_k(x-c_k)}
		\end{equation}
	\end{small}
	where $\psi_0$ is a $C^1$ function, $C_{k,0}:=\frac{\rho_{t_0}}{|(f^{k-1}_{t_0})'(c_1)|^{1/2}}$, $|C_{k,1}|\leq\frac{U_{t_0}}{|(f^{k-1}_{t_0})'(c_1)|^{3/2}}$, $\sigma_k=\text{sgn}(Df^{k-1}_{t_0}(c_1))\in\{\pm\}$ and $w_0,w_1>0$.
	\\Our goal here is to study the regularity of the $\eta$-Marchaud derivative, $0\le\eta<\frac{1}{2}$ of $\rho_{t_0}$, for $t_0$ a Misiurewicz parameter. The main result of this section is:
	\begin{theorem}\label{thm:reginvdens}
		Let $f_{t_0}$ be a Misiurewicz real-analytic unimodal map as in Remark \ref{rem:gensetting}, with invariant density $\rho_{t_0}$. 
		\\For $0\le\eta<\frac{1}{2}$, the Marchaud derivatives $M^\eta_{\pm}[\rho_{t_0}]\in H^s_p(\R)$, for any $0\le s<\frac{1}{2}-\eta$, $1<p<\frac{1}{1/2+\eta+s}$.
	\end{theorem} 
	We introduce the following notation, for $-1\le\beta\le \frac{1}{2}$:
	\begin{equation}
		\left\{
		\begin{aligned}
			&f_{\beta,a,+}(x):=(x-a)^{\beta}\quad &f_{\beta,a,-}(x):=(a-x)^{\beta}\\
			&\tilde f_{\beta,a,A,+}(x):=\mathds 1_{a<x<a+A}(x-a)^{\beta}\quad &\tilde f_{\beta,a,A,-}(x):=\mathds 1_{a-A<x<a}(a-x)^{\beta}
		\end{aligned}
		\right.
	\end{equation}
	\iffalse
	\begin{equation}
		\left\{
		\begin{aligned}
			&f_{-1/2,+}(x):=\mathds 1_{x>a}(x-a)^{1/2}\quad    &f_{-1/2,-}(x):=\mathds{1}_{a<x}(a-x)^{1/2}\\
			&f_{-1/2,+,A}:=\mathds 1_{a<x<a+A}(x-a)^{1/2}\quad &f_{-1/2,-,A}:=\mathds 1_{a-A<x<a}(a-x)^{1/2}.
		\end{aligned}
		\right.
	\end{equation}
	\fi
	We begin by computing relevant fractional integrals of the previously defined functions:
	\begin{lemma}\label{lemme:fracint}
		For any $0\le\eta\le 1$, and $\sigma\in\{\pm\}$ one has
		\begin{equation}\label{eq:fracint}
			I_{a_\sigma}^\eta(f_{\beta,a,\sigma})=\dfrac{\Gamma(\beta+1)}{\Gamma(\beta+1+\eta)}f_{\eta+\beta,a,\sigma}.
		\end{equation} 
		In particular, for $0\le\eta<\frac{1}{2}$,	
		\begin{align*}
			f_{-1/2,a,\sigma}&=I_{a_\sigma}^\eta\left(\dfrac{\Gamma(1/2)}{\Gamma(1/2-\eta)}f_{-1/2-\eta,a,\sigma}\right)\\
			f_{1/2,a,\sigma}&=I_{a_\sigma}^\eta\left(\dfrac{\Gamma(3/2)}{\Gamma(3/2-\eta)}f_{1/2-\eta,a,\sigma} \right)
		\end{align*}
		so that $f_{\pm1/2,a,+}\in I^s(L^p(a,a+A))$, resp. $f_{\pm1/2,a,-}\in I^s(L^p(a-A,a))$, for $0\le s<\frac{1}{2}$, $1<p<\frac{1}{1/2+s}$, and for any $A>0$.
	\end{lemma}
	\begin{proof}
		We focus on the case $\sigma=+$, the other case being similar. By \eqref{def:intervalfracint}, for $x>a$ one has, by the change of variables $t=a+(x-a)s$,
		\begin{align*}
			I_{a+}^\eta(f_{\beta,a,+})(x)&=\dfrac{1}{\Gamma(\eta)}\int_a^x (t-a)^{\beta}(x-t)^{\eta-1}dt\\
			&=\dfrac{1}{\Gamma(\eta)}(x-a)^{\beta+\eta}\int_0^1 s^{\beta}(1-s)^{\eta-1}ds\\
			&=\dfrac{B(\beta+1,\eta)}{\Gamma(\eta)}(x-a)^{\beta+\eta},
		\end{align*}
		where $B(\cdot,\cdot)$ is the Beta function defined by $B(x,y)=\dfrac{\Gamma(x)\Gamma(y)}{\Gamma(x+y)}$ (see, e.g, \cite[Section 1.3 D.]{Samko} for further properties of the Beta function). This gives \eqref{eq:fracint}.
	\end{proof}
	
	This immediately yields the following result:
	\begin{lemma}\label{lemme:regofspike}
		Let $A>0$. One has $\tilde f_{\pm1/2,a,A,\sigma}\in H^s_p(\R)$, for any $0\le s<\frac{1}{2}$, $1<p<\frac{1}{1/2+s}$.
		\\Furthermore, $\|\tilde f_{\pm1/2,a,A,\sigma}\|_{H^s_p}=\|\tilde f_{\pm1/2,0,A,\sigma}\|_{H^s_p}$ is independent of $a$.
	\end{lemma}
	\begin{proof}
		We treat only the case $\sigma=+$, the other case being similar. Recall that $\tilde f_{\pm1/2,a,A,+}=\mathds 1_{a<x<a+A}f_{\pm1/2,a,+}$, hence it is clear that it belongs to $L^p(\R)$ for $1\le p<2$.
		\\By the previous lemma, $f_{\pm1/2,a,+}\in I^s(L^p(a,a+A))$, for any $0\le s<\frac{1}{2}$, $1<p<\frac{1}{1/2+s}$, so that by \cite[Theorem 13.10]{Samko}, $\tilde f_{\pm1/2,a,A,+}\in I^s(L^p(\R))$ for the same range of $s$, $p$. The first part of the lemma then follows from Theorem \ref{thm:fracintfracsobolev}. 
		\\The second part of the lemma is easily seen from the definition of the $H^s_p$ norm and invariance of Lebesgue measure by translation. 
	\end{proof}
	
	\begin{proof}[Proof of Theorem \ref{thm:reginvdens}]
		Functions appearing in the series in Ruelle's decomposition \eqref{eq:invdensdec} are of the form $\tilde f_{\pm1/2,c_k,w_i,\sigma}$. Hence, 
		Lemma \ref{lemme:regofspike} entails that their $H^s_p$ norm are uniformly bounded, and the Collet-Eckmann condition \eqref{CE} insures that the series in \eqref{eq:invdensdec} converges exponentially fast in $H^s_p$ norm. Thus, $\rho_{t_0}\in H^s_p(\R)$ for all $0\le s<\frac{1}{2}$, $1<p<\frac{1}{1/2+s}$, and Proposition \ref{prop:Marchaudboundedop} yields $M_{\pm}^\eta[\rho_{t_0}]\in L^q(\R)\cap I^{s-\eta}(L^p(\R))$ for all $0\le s<\frac{1}{2}$, $1<p<\frac{1}{1/2+s}$ and $q\in\left(\frac{1}{1+\eta},\frac{1}{1/2+\eta+s}\right)$. Taking $p=q$, which is possible in the range $\left(1,\frac{1}{1/2+\eta+s}\right)$, gives the wanted result.
	\end{proof}
	\section{The fractional susceptibility functions}
	In this section, we show our main result: the formal series defining the fractional susceptibility functions \eqref{def:respfracsus}, \eqref{def:frozenfracsus} and \eqref{def:sffracsus} are holomorphic on a disk of radius greater than one.
	\\To alleviate notation we will denote by $I=I_{t_0}$.
	\begin{theorem}\label{thm:regfrozenfracsus}
		Let $f_{t_0}$ be a mixing, Misiurewicz real-analytic unimodal map as in Remark \ref{rem:gensetting}, $\phi\in L^\infty(I)$, $0\le\eta<1/2$. 
		\\Then the response, frozen and semifreddo fractional susceptibility functions $\Psi^{\rsp}_\phi(\eta,.)$, $\Psi^{\fr}_{\phi}(\eta,.)$ and $\Psi^{\smf}_\phi(\eta,.)$ are well-defined and holomorphic in a complex disk $\mathbb D(0,\theta^{-1})$ with $0<\theta<1$.
	\end{theorem}
	Our starting point in this study is the following lemma:
	\begin{lemma}\label{lemme:decayofcorrel}
		For a mixing Misiurewicz real-analytic unimodal map $f_{t_0}$ as in Remark \ref{rem:gensetting}, a $L^\infty(I)$ observable $\phi$ and a $H^s_p(I)$ observable $\psi$ $(s>0, p>1)$, there exists $C=C(t_0,s,p)$ and $0<\theta<1$ such that
		\[\left|\int_I\phi\L_{t_0}^j(\psi)dx-\int_I\phi\rho_{t_0}dx\int_I\psi dx\right|\leq C\theta^j\|\phi\|_{L^\infty}\|\psi\|_{H^s_p}. \]
	\end{lemma}
	\begin{proof}
		It is shown in \cite[Proposition 4.10 and 4.11]{BS12} that for a Misiurewicz parameter\footnote{In fact, the results of \cite{BS12} hold in the more general setting of Topologically Slow Recurrent (TSR) parameters. Those are measure-theoretic generic among stochastic parameters \cite[Rem 2.3]{BS12}. We refer to \cite[Eq(5),Prop. 2.2]{BS12} for a definition and characterization of TSR parameters.} $t_0$, there is a tower extension $\hat f_{t_0}:\hat I\circlearrowleft$ with associated transfer operator $\hat\L_{t_0}$ such that there exists a Banach space $\mathcal B$, a fixed point $\hat\rho_{t_0}\in\mathcal B$ and a measure $\nu$, constants $C>0$ and $\kappa<1$ satisfying: for any $\hat\psi\in\mathcal B$
		\begin{equation*}
			\left\|\hat\L_{t_0}^n(\hat\psi)-\hat\rho_{t_0}\int_{\hat I}\hat\psi d\nu \right\|_{\mathcal B}\le C\kappa^n\|\hat\psi\|_{\mathcal B}.
		\end{equation*}
		Furthermore, there is a bounded operator $\Pi:\mathcal B\to L^1(I)$, such that $\Pi\circ\hat\L_{t_0}=\L_{t_0}\circ\Pi$ (see \cite[Definition 4.5 and p.34]{BS12}). It follows that for $\psi\in C^1$, supported in $I$, we may construct $\hat\psi=(\psi,0,\dots,0,\dots)\in \mathcal B$, that trivially satisfies $\|\hat\psi\|_{\mathcal B}=\|\psi\|_{W^{1,1}}$, $\int_{\hat I}\hat\psi d\nu=\int_I\psi dx$ and $\Pi(\hat\psi)=\psi$. Then one has
		\begin{equation*}
			\left\|\L_{t_0}^n\psi-\rho_{t_0}\int_I\psi dx \right\|_{L^1}\le C\kappa^n \|\psi\|_{W^{1,1}},
		\end{equation*}
		for any $\psi\in C^1$ supported in I. By duality we have, for any $\phi\in L^\infty(I)$
		\begin{equation}
			\left|\int\phi\L_{t_0}^n\psi dx-\int\phi\rho_{t_0}dx\int\psi dx \right|\leq C\kappa^n\|\phi\|_{L^\infty}\|\psi\|_{W^{1,1}}.
		\end{equation}
		To extend this estimate to $\psi\in H^s_p(I)$, $s>0,~p>1$, we notice that since $p>1$, the Sobolev embeddings imply that, for any $\tilde s >2$ (we may choose $\tilde s<2+s$), there exists $\tilde C$ such that for any compactly supported $g\in H^{\tilde s}_p$
		\[\|g\|_{C^1} \le \tilde C  \|g\|_{H^{\tilde s}_p} \, .\]
		Since $s>0$, using mollification, we can approach $\psi$ by $C^1$ functions
		$\psi_\epsilon$ with
		\[
		\|\psi_\epsilon\|_{C^1} \le \tilde C  \| \psi_\epsilon\|_{H^{\tilde s}_p}
		\le  \tilde C_0\|\psi\|_{H^s_p}\epsilon^{-2}\, ,
		\quad
		\|  \psi-\psi_\epsilon\|_{L^p}
		\le \tilde C_1 \epsilon^{s} \|\psi\|_{H^s_p}\, ,
		\]
		for every $\epsilon>0$.
		\\Note that, still by Sobolev embedding, $\|\psi_\epsilon\|_{W^{1,1}}\leq C\|\psi_\epsilon\|_{C^1}$.
		\\The lemma then follows from the following facts.
		\\First,
		$$\int \phi(x) \L_{t_0}^j(\psi_\epsilon) (x)\, dx=
		\int (\phi \circ f_{t_0}^j) \psi_\epsilon \, dx
		\, .
		$$
		Second,
		\[	 \left. 
		\begin{aligned}
			&\left|\int \phi \, d\mu_{t_0} \int (\psi_\epsilon-\psi) \, dx\right|\\
			&\left|\int (\phi \circ f_{t_0}^j) (\psi -\psi_\epsilon )\, dx\right|
		\end{aligned}
		\right\}\le
		\sup|\phi| \| \psi-\psi_\epsilon \|_{L^p}.\]
		To conclude, for each $j$, choose $\epsilon=\kappa^{j/(s+2)}$, so that
		\begin{equation}\label{eq:radius}
			\frac{\kappa^j}{\epsilon^2} =\epsilon^s = \kappa^{js/(s+2)}=: \theta^j\, .
		\end{equation}
	\end{proof}
	\begin{remark}
		Looking at the regularity obtained for the Marchaud derivative in Theorem \ref{thm:reginvdens} and \eqref{eq:radius} we see that 
		\[\kappa^{\frac{1-2\eta}{5-2\eta}}<\theta\le 1. \]
		In particular, when $\eta\to\frac{1}{2}$, we get $\theta\to 1$.
	\end{remark}
	\paragraph{Proof of Theorem \ref{thm:regfrozenfracsus}.}
	Applying Lemma \ref{lemme:decayofcorrel} to $\psi= \mathds{1}_{I}M^\eta(\rho_{t_0})\in H^s_p(I)$ for $0\le\eta<1/2$, we obtain easily that
	$\Psi^{\rsp}_\phi(\eta,.)$ is holomorphic in a disk $\mathbb D(0,\theta^{-1})$.
	To extend this result to the frozen fractional susceptibility function \eqref{def:frozenfracsus}, one may proceed as in \cite[Prop 2.6]{BS20}, showing that the response and frozen susceptibilities differ by a function that is holomorphic in a disk of radius strictly greater than one. 
	\\Instead, we will rely on another method, closer to the one presented in \cite{ABLP20}, which apply to both the frozen and semifreddo fractional susceptibilities. Furthermore, it allows to treat the more general setting \eqref{eq:gensetting} described at the end of the Introduction. 
	\\ We remark that the relation $f_{t_0+t}=f_{t_0}+tX_{t_0}\circ f_{t_0}$ implies that for any $H^s_p(I)$ $(s>0, p>1)$ observable $g$, any $t\in[t_{min},t_{max}]$, any $x\in I$ we have 
	\[(\L_{t_0}g)(x)=(\L_{t_0+t}g)(x+tX_{t_0})\left(1+tX'_{t_0}(x+tX_{t_0})\right).\]
	Note that in the case of the quadratic family \eqref{def:quadfamily}, this last equality simply reads $\L_{t_0}g(x)=\L_{t_0+t}g(x+t)$.
	\iffalse
	Hence, if we introduce the shift operator $\tau_t$, classically defined by $\tau_t(g):=g(.-t)$, we get that 
	\begin{equation}\label{eq:keyid}
		[\L_{t_0+t}-\L_{t_0}]g=[\tau_t-Id]\L_{t_0}g
	\end{equation}
	\fi
	Up to reducing the interval $[t_{min},t_{max}]$, we may assume that $Id+tX_{t_0}$ is a $C^1$ diffeomorphism.
	\\One then has:
	\begin{align*}
		[\L_{t_0+t}-\L_{t_0}]\rho_{t_0}&=\frac{1}{1+tX'_{t_0}}\rho_{t_0}\circ(Id+tX_{t_0})^{-1}-\rho_{t_0}\\
		&=\frac{1}{1+tX'_{t_0}}\left[\rho_{t_0}\circ(Id+tX_{t_0})^{-1}-\rho_{t_0}-tX'_{t_0}\rho_{t_0}\right].
	\end{align*}
	By Theorem \ref{thm:reginvdens}, $\rho_{t_0}\in H^{s}_p$ for $0\le s<\frac{1}{2}$ and $1<p<\frac{1}{1/2+s}$, thus, by Lemma \ref{lemma:regshiftop}, it follows easily that
	\begin{equation}\label{eq:regshiftop}
		\|[\L_{t_0+t}-\L_{t_0}]\rho_{t_0}\|_{H^{\tilde s}_p}\le C|t|^{s-\tilde s}\|\rho_{t_0}\|_{H^s_p},
	\end{equation}
	with $C$ independent on $t$.
	\\For $0<s<\frac{1}{2}$, $1<p<\frac{1}{1/2+s}$, fix $0<\tilde s<s$ such that $0<\eta<s-\tilde s$.
	Applying Lemma \ref{lemme:decayofcorrel} to \[\psi=\mathds{1}_I\frac{[\L_{t_0+t}-\L_{t_0}]\rho_{t_0}}{|t|^{1+\eta}},\] which is in $H^s_p(I)$ for $0<s<\frac{1}{2}$, $1<p<\frac{1}{1/2+s}$, and using \eqref{eq:regshiftop}, we get (note that $\int_I\psi dx=0$)
	\begin{align}\label{eq:goodbound}
		\left|\int_I\phi\L_{t_0}^j\left(\frac{[\L_{t_0+t}-\L_{t_0}]\rho_{t_0}}{|t|^{1+\eta}}\right)dx\right|
		&\leq C\theta^j|t|^{s-\tilde s-1-\eta} \|\phi\|_{L^\infty}\|\rho_{t_0}\|_{H^s_p}
	\end{align}
	
	For our choice of $s,\tilde s$, one may integrate this last bound for $t\in[t_{min},t_{max}]$, to obtain
	\begin{align*}
		\left|\int_{t_{min}}^{t_{max}}\int_I\phi\L_{t_0}^j\left(\frac{[\L_{t_0+t}-\L_{t_0}]\rho_{t_0}}{|t|^{1+\eta}}\right)dxdt\right|&\leq C\theta^j \|\phi\|_{L^\infty}\|\rho_{t_0}\|_{H^s_p}
	\end{align*}
	
	For $t>t_{max}$ (the case $t<t_{min}$ is similar), one gets
	\begin{align*}
		&\left|\int_{t>t_{max}}\int_I\phi\L_{t_0}^j\left(\frac{[\L_{t_0+t}-\L_{t_0}]\rho_{t_0}}{|t|^{1+\eta}}\right)dxdt\right|\\
		&\leq\int_{t>t_{max}}\frac{1}{t^{1+\eta}}dt\left|\int_I\phi\L_{t_0}^j(\L_{t_0+t_{max}}-\L_{t_0})\rho_{t_0}dx\right|\\
		&\leq C\theta^j\eta^{-1}t_{max}^{-\eta}\|\phi\|_{L^\infty}\|\rho_{t_0}\|_{H^s_p}
	\end{align*}
	Hence, by Fubini, the last two bounds implies that for any fixed $0<\eta<\frac{1}{2}$ and $\phi\in L^\infty(I)$, the formal series defining $\Psi^{\fr}_\phi(\eta,.)$ converges on a disk of radius $\theta^{-1}>1$.
	
	For the semifreddo susceptibility function \eqref{def:sffracsus}, one may integrate \eqref{eq:goodbound} for $t\in\Omega\cap[t_{min},t_{max}]$, to obtain
	\begin{align*}
		&\left|\int_{\Omega\cap[t_{min},t_{max}]}\int_I\phi\L_{t_0}^j\left(\frac{[\L_{t_0+t}-\L_{t_0}]\rho_{t_0}}{|t|^{1+\eta}}\right)dxdt\right|\\
		&\le C\theta^j\|\phi\|_{L^\infty}\|\rho_{t_0}\|_{H^s_p}\int_{t_{min}}^{t_{max}}\mathds{1}_{\Omega}|t|^{s-\tilde s-1-\eta}dt\\
		&\le C\theta^j\|\phi\|_{L^\infty}\|\rho_{t_0}\|_{H^s_p}\int_{t_{min}}^{t_{max}}|t|^{s-\tilde s-1-\eta}dt
	\end{align*}
	and thus we may conclude as in the previous case. For $t>t_{max}$, resp. $t<t_{min}$, we proceed similarly.
	$\hfill\square$
	\begin{small}
		\bibliography{biblio}
		\bibliographystyle{plain}
	\end{small}
\end{document}